\documentclass[a4paper,12pt]{amsart}
\usepackage[utf8]{inputenc}
\usepackage[T1]{fontenc}
\usepackage{amsthm}
\usepackage{amssymb}
\usepackage{listings}

\title{Ley de Reciprocidad Cuadrática y aplicaciones}
\author{Mario Pérez Maletzki}
\date{2021}

\newtheorem{teo}{Teorema}[section]
\newtheorem{prop}[teo]{Proposición}
\newtheorem{defi}[teo]{Definición}
\newtheorem{cor}[teo]{Corolario}
\newtheorem{ej}[teo]{Ejemplos.}
\newtheorem{lema}[teo]{Lema}

\begin{document}
\maketitle	
	\begin{abstract}
		El objetivo de este trabajo es introducir todos los conceptos y resultados necesarios para finalmente poder dar una demostración rigurosa de la Ley de Reciprocidad Cuadrática y ver como podemos aplicarla para obtener resultados de teoría de números que distan mucho de ser triviales tales como el problema de los dos cuadrados, el problema de determinar cuándo una ecuación en congruencias de segundo grado tiene solución y otros. 
		
		Para motivar algunos resultados y estimular la intuición sobre ellos incluimos ejemplos prácticos usando el programa GAP.
	\end{abstract}

\section{Introducción}

El desarrollo y descubrimiento de la Ley de Reciprocidad Cuadrática fue muy lento e involucró a muchos matemáticos tales como Gauss, Euler y Legendre entre otros. Sus orígenes se remontan a cuestiones sobre ecuaciones diofánticas y el problema que escribió Fermat a Mersenne en una carta en la cual afirmaba:

 \emph{“Todo número primo, que supere por una unidad un múltiplo de 4, es una única vez la suma de cuadrados, y es una única vez la hipotenusa de
un triángulo rectángulo”} 

Fermat, como era costumbre en él, no dio una demostración de este enunciado y hubo que esperar a Euler para que la proporcionara pasados unos años. 

Euler se interesó pues en la teoría de números, y como veremos formuló un criterio muy útil para determinar cuando un entero era un residuo cuadrático módulo un primo dado y conjeturó enunciados muy similares al que afirma la Ley de Reciprocidad Cuadrática, pero no fue hasta que llegó Gauss, quien a sus 19 años enunció y probó este teorema. Mientras Euler se preguntaba si dado un número como módulo, otro número era residuo cuadrático de éste, Gauss planteó el problema inverso: dado un número entero, ¿sobre qué enteros es este un residuo cuadrático?
Es por esto que se adoptó el nombre de Ley de Reciprocidad Cuadrática, al cual Gauss denominó el \emph{Theorema Aureum}. Gauss dio $8$ demostraciones distintas de este teorema a lo largo de su vida, y a día de hoy este es uno de los teoremas con más demostraciones distintas de la historia de las Matemáticas.

\section{Ley de Reciprocidad Cuadrática}

Comenzamos definiendo qué es un residuo cuadrático.

\begin{defi}[Residuo Cuadrático]
	Sean $m\in \mathbb{N}$ mayor o igual que 2 y  $n$ un entero coprimo con $m$. Diremos que $m$ es un residuo cuadrático módulo $m$ si tiene solución la siguiente congruencia: $$ x^2\equiv n \ (mod \ m).$$
\end{defi}

De la definición es claro que si $n'$ es otro entero tal que $n\equiv n' \ (mod\ m)$, entonces $n$ será un residuo cuadrático módulo $m$ si y sólo si lo es $n'$.

\underline{Observación:} Si $x$ es una solución de la anterior ecuación, cualquier $y$ congruente con $x$ módulo $m$ también lo será, pues si $x\equiv y \ (mod\ m) \Rightarrow x^2\equiv y^2 \ (mod\ m) $.

Esta observación es muy útil, pues nos indica que debemos buscar las soluciones en el conjunto $\{1,\ldots, m-1\}$.
\begin{prop}
	Dado un entero positivo $m$, en el conjunto $\smash{\{0,1,\ldots, m-1\}}$ pueden haber como máximo $\frac{m}{2}+1$ residuos cuadráticos si $m$ es par y $\frac{m+1}{2}$ si $m$ es impar.
\end{prop}
\begin{proof}[Demostración]
Primero observamos que $0$ es trivialmente un residuo cuadrático (basta tomar $x=m^2$ por ejemplo). 

Estudiamos de entre los números $\{1^2,2^2,\ldots (m-1)^2\}$ cuántos son congruentes entre sí módulo $m$.

Supongamos que $m$ es impar. En este caso $m-1$ es par, y como además $(m-1)^2\equiv 1^2\ (mod\ m)$, $(m-2)^2\equiv 2^2\ (mod\ m)$, ... concluimos que como máximo pueden haber $1+\frac{m-1}{2}=\frac{m+1}{2}$ residuos cuadráticos distintos.
Si $m$ es par, de forma análoga tenemos que $(m-1)^2\equiv 1^2\ (mod\ m)$, $\smash{(m-2)^2\equiv 2^2\ (mod\ m)}$, ...,$(\frac{m}{2}+1)^2\equiv (\frac{m}{2}-1)^2\ (mod\ m)$ y obtenemos que pueden haber como máximo $1+\frac{m-2}{2} + 1= \frac{m}{2}+1$ residuos cuadráticos.
\end{proof}

\underline{Nota:} De ahora en adelante supondremos que el módulo es tal que $m>2$, pues es un caso sin interés y que rompe con la simetría de algunos teoremas. 

Cuando tratamos con congruencias módulo un número primo podemos refinar el anterior resultado:

\begin{teo}
	Sea $p$ un número primo. Entonces exactamente la mitad de los elementos de $\{1,2,\ldots (p-1)\}$ son residuos cuadráticos módulo $p$.
\end{teo}
\begin{proof}[Demostración]
	Vamos a probar que ningún par de los números $1^2,2^2,\ldots (\frac{p-1}{2})^2$ pueden ser congruentes entre sí módulo $p$, y por tanto, todos ellos serás residuos cuadráticos distintos módulo $p$ y como por la proposición anterior no pueden haber más de $\frac{p-1}{2}$, habremos terminado.
	
	Si $1\leq x<y\leq \frac{p-1}{2}$ son tales que $x^2\equiv y^2 \ (mod\ p) $ entonces $\smash{p| (x-y)(x+y)}$ pero como $x+y<p$, esto sólo puede ser posible si $x=y$. 
\end{proof}

\begin{ej}

\begin{lstlisting}

[language=GAP]
gap> Z13:= Integers mod 13;
GF(13)
gap> List(Elements(Z13),x_>Int(x^2));
[ 0, 1, 4, 3, 12, 9, 10, 1, 4, 3, 12, 9, 10 ]


\end{lstlisting}

Observamos que hay $6$ residuos cuadráticos módulo $13$ entre $\{1,2,\ldots 12\}$, lo cual sabíamos que debía ocurrir por ser $13$ primo.

\begin{lstlisting}[language=GAP]
gap> Z15:= Integers mod 15;

(Integers mod 15)

gap> List( Elements (Z15) , x _> Int(x^2)  );

[ 0, 1, 4, 9, 1, 10, 6, 4, 4, 6, 10, 1, 9, 4, 1 ]

gap>
\end{lstlisting}

En este caso (siendo $15$ un número compuesto) nos encontramos con que sólo hay $5$ residuos cuadráticos módulo $15$ en $\{1,2,\ldots 14\}$.

\end{ej}

\begin{teo}\label{prod-residuos}
	Si  $p$ es un número primo y $a$ y $b$ son enteros coprimos con $p$, entonces si ambos son residuos cuadráticos módulo $p$, su producto $a\cdot b$ también es un residuo cuadrático módulo $p$. Si uno de ellos lo es pero el otro no entonces su producto tampoco lo es, y si ninguno de ellos lo es su producto sí lo es. 
\end{teo}
\begin{proof}[Demostración]
	Si ambos lo son, deben existir enteros $x$ e $y$ tales que $$x^2\equiv a \ (mod \ p)$$ $$y^2\equiv b \ (mod \ p)$$ y por tanto tenemos que $$(xy)^2\equiv x^2\cdot y^2\equiv a\cdot b \ (mod \ p)$$ de lo cual concluimos que $ab$ es un residuo cuadrático módulo $p$.
	
	Supongamos que $a$ lo es pero $b$ no. Por serlo $a$, existirá un $x$ entero tal que $x^2\equiv a \ (mod \ p)$, y si $ab$ también lo fuera existiría otro entero $z$ tal que $z^2\equiv ab \ (mod \ p)$, pero entonces, teniendo en cuenta que por ser $x$ coprimo con $p$ tiene inverso en $\mathbb{Z}_p$, tendríamos que $$z^2\equiv ab\equiv x^2b \ (mod \ p)$$ y por tanto que $$(z\cdot x^{-1})^2\equiv z^2\cdot (x^{-1})^2\equiv b \ (mod \ p)$$ lo cual es una contradicción, pues habíamos supuesto que $b$ no es residuo cuadrático.
	
	Finalmente si ni $a$ ni $b$ son residuos cuadráticos módulo $p$, multiplicamos $a$ por cada elemento de $\{1,\ldots,p-1\}$ que sí sea residuo cuadrático y obtenemos un total de $\frac{p-1}{2}$ no-residuos cuadráticos; pues en dicho conjunto sabemos que hay $\frac{p-1}{2}$ residuos cuadráticos y que al multiplicarlos por el no-residuo $a$ obtenemos un no-residuo, y si $a\cdot x\equiv a\cdot y\ (mod\ p)$ entonces multiplicando en ambos lados por el inverso de $a$ en $\mathbb{Z}_p$ llegamos a que $x\equiv  y\ (mod\ p)$. Por tanto como hay justamente $\frac{p-1}{2}$ no-residuos cuadráticos, necesariamente $a\cdot b$ tiene que ser un residuo cuadrático.
	
\end{proof}

\begin{defi}[Símbolo de Legendre]
	Para cada número primo impar $p$ y cada entero $n$ primo con $p$, definimos el símbolo de Legendre de $n$ respecto de p como $$\left( \frac{n}{p}\right):=\left\{
	\begin{array}{cl}
	1&\mbox{ si n es un resíduo cuadrático módulo p},\\
	-1&\mbox{si n no es resíduo cuadrático módulo p}. 
	\end{array}\right.
	$$
\end{defi}
\begin{teo}[Teorema de Wilson]
Sea $p$ un entero mayor que 1. Entonces $p$ es primo si y sólo si $$(p-1)!\equiv -1\ (mod\ p).$$

\end{teo}
\begin{proof}[Demostración]
	($\Rightarrow$) Supongamos que $p$ es un número primo. Sabemos que entonces $\mathbb{Z}_p$ es cuerpo, y en particular que todo elemento no nulo tiene inverso respecto a la multiplicación (y que es único). Además, sólo hay dos números que sean inversos de sí mismos: en efecto, supongamos que $1\leq x\leq p-1$ es tal que $x\equiv x^{-1}\ (mod \ p)$ o lo que es lo mismo, que $x^2\equiv 1 \ (mod\ p)$. Entonces, $(x-1)(x+1)$ es divisible por $p$, y como $p$ es primo esto sólo puede ocurrir si bien $x=1$ o bien $x=p-1$.
	Por tanto, si multiplicamos todos los elementos de $\{1,2,\ldots,p-1\}$, agrupando dos a dos cada uno con su inverso tenemos que $$(p-1)!\equiv 1\cdot 1\cdot\ldots \cdot 1 \cdot (p-1)\equiv (p-1)\equiv -1 \ (mod\ p).$$
	
	($\Leftarrow$) Supongamos que $p$ es un entero que cumple que $(p-1)!\equiv -1\ (mod\ p)$. Supongamos que $p$ no es primo y sea $1<c<p$ un divisor propio de $p$. Obviamente $c$ divide a $(p-1)!$, y como por hipótesis $(p-1)!\equiv -1\ (mod\ p)$ esto quiere decir que $p$ divide a $(p-1)!+1$ y al ser $c$ es un divisor de $p$, $c$ será también un divisor de $(p-1)!+1$. Pero esto es imposible, pues el único entero que divide a dos números consecutivos es el 1 y habíamos supuesto que $1<c<p$.
\end{proof}
\begin{teo}[Criterio de Euler]
Sea $p$ un primo impar y $a$ un entero coprimo con $p$. Entonces $$a^{\frac{p-1}{2}}\equiv \left\{
\begin{array}{cl}
1\ (mod\ p)&\mbox{ si n es un resíduo cuadrático módulo p},\\
-1\ (mod\ p)&\mbox{si n no es resíduo cuadrático módulo p}. 
\end{array}\right.
$$
\end{teo}
\begin{proof}[Demostración]
	Consideramos los pares de elementos $(x,y)$ tales que $\smash{x \leq y \leq p-1}$ y
	$x\cdot y \equiv a \ (mod\ p)$. Observamos que, para cada $x$, existe un único
	$y\in\mathbb{Z}_p$ tal que $x\cdot y \equiv a \ (mod\ p)$, pues al ser $p$ primo y $x$ coprimo con $p$, el inverso de $x$ en $\mathbb{Z}_p$ existe, y $x\cdot(x^{-1}a)\equiv a\ (mod\ p)$. Además si también 	$x\cdot z \equiv a \ (mod\ p)$ entonces $x\cdot y \equiv x\cdot z \ (mod\ p)$ y multiplicando por el inverso de $x$ en ambos lados llegaríamos a que $y \equiv z \ (mod\ p)$ (y al ser $1\leq y,z\leq p-1$ esto sólo puede ocurrir si $y=z$).
	
	Distinguimos pues dos casos. Si $a$ no es residuo cuadrático, los elementos que forman cada uno de los pares posibles tienen que ser necesariamente distintos entre sí, y como hay $p-1$ elementos habrán $\frac{p-1}{2}$ pares. Si los multiplicamos todos ellos, por el teorema de Wilson obtenemos que $$(p-1)!\equiv a^{\frac{p-1}{2}}\equiv -1\ (mod\ p).$$
	
	En el caso en que $a$ sí sea un residuo cuadrático, habrán exactamente dos enteros distintos, que representaremos por $\sqrt{a}$ y $\sqrt{-a}$, que sean solución de la ecuación $$x^2\equiv a \ (mod \ p).$$ 
	Por tanto habrán $\frac{p-3}{2}$ pares formados por elementos distintos y dos pares que son $(\sqrt{a},\sqrt{a})$ y $(\sqrt{-a},\sqrt{-a})$.
	
	Multiplicamos ahora todos los elementos de todos los pares obtenemos que $$ a^{\frac{p+1}{2}}\equiv (p-1)!\cdot \sqrt{a}\cdot\sqrt{-a}\equiv (-1)\cdot (-a)\equiv a\ (mod\ p)$$
	
	y multiplicando en ambos lados por $a^{-1}$ obtenemos que $$ a^{\frac{p-1}{2}}\equiv 1\ (mod\ p).$$
\end{proof}

\underline{Observación:} Con el criterio de Euler podemos demostrar el teorema (\ref{prod-residuos}) de forma más directa teniendo en cuenta que $$(ab)^{\frac{p-1}{2}}=a^{\frac{p-1}{2}}\cdot b^{\frac{p-1}{2}}$$ y que por tanto $ab$ es residuo cuadrático si y sólo si $(ab)^{\frac{p-1}{2}} \equiv1\ (mod\ p)$ si y sólo si, o bien $$a^{\frac{p-1}{2}} \equiv1\ (mod\ p)\ \mbox{  y  }\ b^{\frac{p-1}{2}} \equiv1\ (mod\ p)$$ o bien  $$a^{\frac{p-1}{2}} \equiv-1\ (mod\ p)\ \mbox{ y }\ b^{\frac{p-1}{2}} \equiv-1\ (mod\ p).$$

El siguiente corolario será de vital importancia cuando en (\ref{primo}) caractericemos qué números primos pueden expresarse como suma de dos números cuadrados.

\begin{cor}\label{-1 res}
	Si $p$ es un primo impar $$\left( \frac{-1}{p}\right)=\left\{
	\begin{array}{cl}
	1&\mbox{ si } p\equiv 1\ (mod\ 4)\\
	-1&\mbox{ si } p\equiv 3\ (mod\ 4)
	\end{array}\right.
	$$
\end{cor}
\begin{proof}[Demostración]
	
	En efecto, por el criterio de Euler $\left( \frac{-1}{p}\right)=(-1)^{\frac{p-1}{2}}$ y ahora basta con tener en cuenta que $\frac{p-1}{2}$ es par si $p\equiv 1\ (mod\ 4)$ e impar si $\smash{p\equiv 3\ (mod\ 4)}$.
	
\end{proof}
\begin{lema}[Lema de Gauss]\label{lema-gauss}
	Sean $p$ un primo impar y $n$ un entero coprimo con $p$. Definimos el conjunto $$S:=\{n,2n,\ldots \frac{p-1}{2}n \}$$ y $S'$ como los representantes de las clases de equivalencia de los elementos de $S$ en $\mathbb{Z}_p$. 
	
	Si denotamos por $k$ el número de elementos de $S'$ que son mayores que $p/2$ entonces $$\left( \frac{n}{p}\right)=(-1)^k.$$
\end{lema}
\begin{proof}[Demostración]
	Primero observamos que en $S$ no hay ningún múltiplo de $p$, y que en $S'$ no hay ningún par de elementos congruentes entre sí módulo $p$, pues si $nx\equiv ny \ (mod \ p)$ multiplicamos en ambos lados por el inverso de $n$ en $\mathbb{Z}_p$ y obtendremos que $x\equiv y \ (mod \ p)$, lo cual, siendo $1\leq x,y\leq \frac{p-1}{2}$ sólo es posible si $x=y$. Por tanto en $S'$ hay exactamente $\frac{p-1}{2}$ elementos.
	
	Si denotamos por $r_1,\ldots,r_l$ a los elementos de $S'$ menores que $\frac{p}{2}$ y por $s_1,\ldots,s_k$ a los mayores que $\frac{p}{2}$, se tiene que $$\{1,2\ldots,\frac{p-1}{2}\}=\{r_1,\ldots,r_l,p-s_1,\ldots,p-s_k\}$$
	
	Para probar la anterior igualdad, observamos que $1\leq r_i, p-s_j\leq \frac{p-1}{2}$ y que por tanto basta con comprobar que todos ellos son incongruentes entre sí. Ya hemos visto que si $i\neq j$, $r_i$ no puede ser congruente con $r_j$; análogamente $s_i$ no puede ser congruente con $s_j$ y por tanto $p-s_i$ no puede ser congruente con $p-s_j$. Si existieran elementos tales que $r_i\equiv p-s_j$, llegaríamos a que $p$ divide a un número de la forma $n(x+y)$ con $1\leq x,y\leq \frac{p-1}{2}$, lo cual es imposible.
	
	Ahora, si multiplicamos todos los elementos de cada conjunto llegamos a que $$(\frac{p-1}{2})!\equiv \prod_{i=1}^{l}r_i \prod_{j=1}^{k}(p-s_j)\equiv (-1)^k \prod_{i=1}^{l}r_i \prod_{j=1}^{k}s_j\ (mod\ p).$$
	
	Por otro lado, como los elementos de $S$ son congruentes uno a uno con los de $S'$, multiplicándolos todos entre sí obtenemos que $$n^{\frac{p-1}{2}}(\frac{p-1}{2})!\equiv \prod_{i=1}^{l}r_i \prod_{j=1}^{k}s_j\ (mod\ p)$$ y juntando todo, después de multiplicar por el inverso de $(\frac{p-1}{2})!$ (el cual existe por ser coprimo con $p$) concluimos que $$n^{\frac{p-1}{2}}(-1)^k\equiv 1\ (mod\ p)$$ o equivalentemente $$n^{\frac{p-1}{2}}\equiv (-1)^k\ (mod\ p).$$
	El resultado se sigue ahora de aplicar el criterio de Euler.
\end{proof}
\begin{cor}
	Si $p$ es un primo impar entonces $$\left( \frac{2}{p}\right)=(-1)^{\frac{p^2-1}{8}}.$$
\end{cor}
\begin{proof}[Demostración]
	Usaremos el lema de Gauss y para ello bastará con estudiar la paridad del conjunto de números mayores que $p/2$ entre el conjunto $\{2,4,\ldots,p-1\}$. 
	Como hay $[\frac{p}{4}]$ números pares menores que $\frac{p}{2}$, habrán $\frac{p-1}{2}-[\frac{p}{4}]$ mayores.
	
	Distinguiremos 4 casos posibles: 
	\begin{enumerate}
		\item Si $p\equiv 1\ (mod \ 8)$, entonces $\frac{p-1}{2}-[\frac{p}{4}]$ será de la forma $4n+2n$ para algún entero $n$. Por tanto $k$ será par.
		\item Si $p\equiv 3\ (mod \ 8)$, entonces $\frac{p-1}{2}-[\frac{p}{4}]$ será de la forma $4n+1+2n$ para algún entero $n$. Por tanto $k$ será impar.
		\item Si $p\equiv 5\ (mod \ 8)$, entonces $\frac{p-1}{2}-[\frac{p}{4}]$ será de la forma $4n+2+2n+1$ para algún entero $n$. Por tanto $k$ será impar.
		\item Si $p\equiv 7\ (mod \ 8)$, entonces $\frac{p-1}{2}-[\frac{p}{4}]$ será de la forma $4n+3+2n+1$ para algún entero $n$. Por tanto $k$ será par. 
	\end{enumerate}

Todo esto lo expresamos de una forma más compacta diciendo que $$\left( \frac{2}{p}\right)=(-1)^{\frac{p^2-1}{8}}$$

pues $\frac{p^2-1}{8}$ es par si $p\equiv 1,7\ (mod \ 8)$ y es impar si $p\equiv 3,5\ (mod \ 8)$.
\end{proof}

\begin{ej}

	\begin{lstlisting}[language=GAP]
	
	
	gap> IsPrime(101);
	true
	gap> 101 mod 8;
	5
	gap> 2^50 mod 101;
	100
	
	\end{lstlisting}

	\begin{lstlisting}[language=GAP]

	gap> IsPrime(41);
	true
	gap> 41 mod 8;
	1
	gap> 2^20 mod 41;
	1
	

	\end{lstlisting}
\end{ej}

Antes de probar el resultado principal, veamos unos ejemplos sugerentes:

\begin{ej}

	\begin{lstlisting}[language=GAP]
	
	
	gap> IsPrime(911);
	true
	gap> 911 mod 4;
	3
	gap> IsPrime(919);
	true
	gap> 919 mod 4;
	3
	gap> IsPrime(929);
	true
	gap> 929 mod 4;
	1
	gap> IsPrime(937);
	true
	gap> 937 mod 4;
	1
	
	gap> 911^459 mod 919;
	918
	
	gap> 919^455 mod 911;
	1
	
	\end{lstlisting}

	Observamos que $911$ no es residuo cuadrático módulo $919$, pero $919$ sí que es residuo cuadrático módulo $911$.

	\begin{lstlisting}[language=GAP]
	
	
	gap> 929^455 mod 911;
	1
	gap> 911^464 mod 929;
	1
	
	
	
	\end{lstlisting}
	
	En este caso tenemos que $929$ es residuo cuadrático módulo $911$ y también $911$ es residuo cuadrático módulo $929$.
	
	\begin{lstlisting}[language=GAP]
	
	gap> 929^459 mod 919;
	1
	gap> 919^464 mod 929;
	1
	
	
	
	
	\end{lstlisting}
	
	Otra vez tenemos que $929$ es residuo cuadrático módulo $919$ y también $919$ es residuo cuadrático módulo $929$.
	
\end{ej}

Vayamos pues con el teorema fundamental:

\begin{teo}[\textbf{Ley de Reciprocidad Cuadrática}]
	Si $p$ y $q$ son números primos (impares) distintos se tiene que $$\left( \frac{p}{q}\right)\left( \frac{q}{p}\right)=(-1)^{\frac{(p-1)(q-1)}{4}}.$$
\end{teo}
\begin{proof}[Demostración]
	Denotamos por $[\frac{n}{p}]$ a la parte entera de $\frac{n}{p}$, es decir, el cociente de la división entera de $n$ entre $p$.
	
	Primero vamos a probar que si $n$ es un número impar coprimo con $p$ y denotamos por $\rho:=\sum_{j=1}^{\frac{p-1}{2}}\left[\frac{jn}{p}\right]$, entonces  $$\left( \frac{n}{p}\right)=(-1)^\rho.$$
	
	Definimos $S'$ como en (\ref{lema-gauss}) y volvemos a denotar por $r_1,\ldots,r_l$ a los elementos de $S'$ menores que $\frac{p}{2}$ y por $s_1,\ldots,s_k$ a los mayores que $\frac{p}{2}$. Para cada $j$ está claro que $jn=\left[\frac{jn}{p}\right]p+t$ para cierto $t\in S'$ y por tanto $$\sum_{j=1}^{\frac{p-1}{2}}jn= \sum_{j=1}^{\frac{p-1}{2}}\left[\frac{jn}{p}\right]p + \sum_{j=1}^{l}r_j +  \sum_{j=1}^{k}s_j.$$ Por otra parte, como ya probamos en (\ref{lema-gauss}) se tiene que $$\{1,2\ldots,\frac{p-1}{2}\}=\{r_1,\ldots,r_l,p-s_1,\ldots,p-s_k\}$$ y por tanto $$\sum_{j=1}^{\frac{p-1}{2}}j= \sum_{j=1}^{l}r_j +  \sum_{j=1}^{k}(p-s_j)= \sum_{j=1}^{l}r_j + kp - \sum_{j=1}^{k}s_j.$$ y restando estas dos expresiones obtenemos que  $$(n-1)\sum_{j=1}^{\frac{p-1}{2}}j=p\left(\sum_{j=1}^{\frac{p-1}{2}}\left[\frac{jn}{p}\right]-k\right) + 2\left(\sum_{j=1}^{k}s_j\right).$$
	
	Observamos que $k$ y $\rho=\sum_{j=1}^{\frac{p-1}{2}}\left[\frac{jn}{p}\right]$ tienen que tener la misma paridad, pues al ser $n$ impar la expresión de la izquierda es par y por ser $p$ un primo impar  $p\left(\sum_{j=1}^{\frac{p-1}{2}}\left[\frac{jn}{p}\right]-k\right)$ es par si y sólo si $k$ y $\rho=\sum_{j=1}^{\frac{p-1}{2}}\left[\frac{jn}{p}\right]$ tienen la misma paridad. 
	
	Por tanto, por el Lema de Gauss (\ref{lema-gauss}) tenemos que $$\left( \frac{n}{p}\right)=(-1)^k=(-1)^\rho.$$ 
	
	A continuación vamos a probar que $$\sum_{j=1}^{\frac{p-1}{2}}\left[\frac{jq}{p}\right]+\sum_{j=1}^{\frac{q-1}{2}}\left[\frac{jp}{q}\right]=\frac{(p-1)(q-1)}{4}$$ lo cual terminará la prueba pues entonces tendremos que $$\left( \frac{p}{q}\right)\left( \frac{q}{p}\right)=(-1)^{\sum_{j=1}^{\frac{p-1}{2}}\left[\frac{jq}{p}\right]}(-1)^{\sum_{j=1}^{\frac{q-1}{2}}\left[\frac{jp}{q}\right]}=(-1)^{\sum_{j=1}^{\frac{p-1}{2}}\left[\frac{jq}{p}\right]+\sum_{j=1}^{\frac{q-1}{2}}\left[\frac{jp}{q}\right]}=(-1)^{\frac{(p-1)(q-1)}{4}}.$$
	
	Supongamos sin pérdida de generalidad que $q<p$ y calculemos la suma $$\rho=\sum_{j=1}^{\frac{p-1}{2}}\left[\frac{jq}{p}\right].$$
	
	Para $j=1$ tenemos que $[\frac{jq}{p}]=0$ y para $j=\frac{p-1}{2}$ tenemos que $$\left[\frac{\frac{p-1}{2}q}{p}\right]=\left[\frac{\frac{q-1}{2}p+\frac{p-q}{2}}{p}\right]=\frac{q-1}{2}$$ y por tanto todos los sumandos toman valores de forma creciente entre $0$ y $\frac{q-1}{2}$. Teniendo en cuenta que por ser $p>q\rightarrow \frac{p-1}{2}\geq\frac{q+1}{2} $ y como además los términos $\frac{jq}p$ están igualmente espaciados, tenemos que para cada $n$ tal que $0\leq n\leq \frac{q-1}{2}$ habrá algún sumando que tome ese valor y para calcular $\rho$ sólo nos hará falta ver cuántos sumandos hay que tomen el mismo valor para cada $n$.
	
	 Para ver esto, si consideramos dos sumandos consecutivos de forma que $$\left[\frac{jq}{p}\right]=n-1\quad \mbox{ y }\quad \left[\frac{(j+1)q}{p}\right]=n$$ se tiene que $$\frac{jq}{p}<n<\frac{(j+1)q}{p} \rightarrow j<\frac{np}{q}<j+1\rightarrow \left[\frac{np}{q}\right]=j.$$ 
	 
	 Por tanto el número exacto de sumando en $\rho$ que toman el valor $n$ será $$\left[\frac{(n+1)p}{q}\right]-\left[\frac{np}{q}\right] $$ y así $$\sum_{j=1}^{\frac{p-1}{2}}\left[\frac{jq}{p}\right]= 1\left(\left[\frac{2p}{q}\right]-\left[\frac{p}{q}\right]\right)+2\left(\left[\frac{3p}{q}\right]-\left[\frac{2p}{q}\right]\right)+$$$$+\ldots+\frac{q-1}{2}\left(\frac{p-1}{2}-\left[\frac{\frac{q-1}{2}p}{q}\right]\right)=-\sum_{k=1}^{\frac{q-1}{2}}\left[\frac{kp}{q}\right]+ \frac{(p-1)(q-1)}{4}$$ y así tendremos que $$\sum_{j=1}^{\frac{p-1}{2}}\left[\frac{jq}{p}\right]+\sum_{k=1}^{\frac{q-1}{2}}\left[\frac{kp}{q}\right]= \frac{(p-1)(q-1)}{4}.$$
\end{proof}

Esta es la forma moderna en la que se enuncia la Ley de Reciprocidad Cuadrática, pero la manera original (obviamente equivalente) de Gauss fue la siguiente:

\emph{Sean $p$ y $q$ dos números primos impares distintos entre sí. Entonces: 
\begin{enumerate}
	\item Si $ p\equiv 1\ (mod\ 4)$, $q$ es un residuo cuadrático módulo $p$ si y sólo si $p$ es un residuo cuadrático módulo $q$.
	\item Si $ p\equiv 3\ (mod\ 4)$, $q$ es un residuo cuadrático módulo $p$ si y sólo si $-p$ es un residuo cuadrático módulo $q$.
\end{enumerate}}

El Criterio de Euler es excelente a nivel teórico, pero computacionalmente no es muy óptimo ya que para determinar si un número es un residuo cuadrático módulo un número primo que sea "grande" tenemos que calcular potencias de un orden muy alto.

 Por ejemplo, para determinar si $19$ es un residuo cuadrático módulo $859$, tendríamos que calcular $$19^{429}\ ( mod\ 859).$$

Sin embargo, con la Ley de Reciprocidad Cuadrática, puesto que tanto $19$ como $859$ con congruentes con $3$ módulo $4$, podríamos determinar si $19$ es un residuo cuadrático módulo $859$ de forma mucho más sencilla; basta con ver si $859$ es un residuo cuadrático módulo $19$, y como $859\equiv 4\ (mod\ 19)$ el problema se reduce a saber si $4$ es un residuo cuadrático módulo $19$ y aquí el Criterio de Euler es más eficiente al calcular $$4^9\ (mod\ 19)=1.$$

Por tanto, podemos afirmar que $19$ \underline{no} es residuo cuadrático módulo $859$.

\begin{ej}

	\begin{lstlisting}[language=GAP]
	
	
	
	gap> IsPrime(881);
	true
	gap> IsPrime(877);
	true
	gap> 877 mod 4;
	1
	gap> 877^440 mod 881;
	1
	
	
	\end{lstlisting}
	
	Según la Ley de Reciprocidad Cuadrática, en este caso debe darse que $881$ es un residuo cuadrático módulo $877$. Lo comprobamos con el Criterio de Euler y efectivamente:
	
	\begin{lstlisting}[language=GAP]
	
	gap> 881^438 mod 877;
	1
	
	
	
	
	\end{lstlisting}

\end{ej}

\section{Aplicaciones}

Veamos algunas aplicaciones de la Ley de Reciprocidad Cuadrática en teoría de números.

\subsection{Enteros que son suma de dos cuadrados}

Nuestro punto de partida es la observación de que conjunto de enteros que son suma de dos cuadrados es cerrado por multiplicación, esto es, dados dos números cada uno de los cuales es suma de dos cuadrados, su producto también podrá ser expresado como suma de dos cuadrados, como muestra la siguiente fórmula: $$(a^2+b^2)(c^2+d^2)=(ac-bd)^2+(ad+bc)^2$$ para cualesquiera enteros $a,b,c$ y $d$.

Parece lógico pues comenzar caracterizando los números primos que pueden escribirse como suma de dos cuadrados y como trivialmente $2=1^2+1^2$ nos centramos en los primos impares. 

\begin{teo}\label{primo}
	Si $p$ es un primo impar, entonces $p$ es suma de dos cuadrados si y sólo si $ p\equiv 1\ (mod\ 4)$.
	
\end{teo}
\begin{proof}[Demostración]
	La suficiencia es elemental, pues dado un número entero, su cuadrado siempre es congruente con 0 o 1 módulo 4 (de comprobación inmediata) y por tanto si un entero es suma de dos cuadrados será congruente con 0,1 o 2 módulo 4. Ahora al ser $p$ un primo impar no puede ser congruente con 0 ni 2 módulo 4 (pues si no sería par) y por tanto si $p$ es suma de dos cuadrados necesariamente debe ser $$ p\equiv 1\ (mod\ 4).$$ 
	
	Supongamos ahora que $ p\equiv 1\ (mod\ 4)$. Por (\ref{-1 res}) tenemos que $\left( \frac{-1}{p}\right)=1$ y por tanto existe un entero $u\in\mathbb{Z}$ tal que $ u^2\equiv -1\ (mod\ p)$.
	
	Consideremos el conjunto de enteros de la forma $x+uy$ tales que $x,y\in\mathbb{Z}$ y $0\leq x,y\leq \sqrt{p}$. 
	Como hay $([\sqrt{p}]+1)^2>p$ posibles pares $(x,y)$ distintos en estas condiciones, por el Principio del Palomar (también llamado \emph{Principio del Casillero}) necesariamente deben de haber $(x_1,y_1)\neq (x_2,y_2)$ tal que $$x_1+uy_1\equiv x_2+uy_2\ (mod \ p)$$ lo cual equivale a que $$x_1-x_2\equiv u(y_1-y_2)\ (mod \ p).$$
	
	Definimos $a:=x_1-x_2$ y $b:=y_1-y_2$. Se tiene que $|a|<\sqrt{p}$,  $|b|<\sqrt{p}$ y $a\equiv ub\ (mod\ p)$. Por tanto, teniendo en cuenta que $ u^2+1\equiv 0\ (mod\ p)$ y operando: $$a^2+b^2\equiv (u^2+1)b^2\equiv 0 \ (mod\ p).$$
	
	Finalmente, nos damos cuenta de que, por un lado $a^2+b^2<2p$ y por otro, al ser $(x_1,y_1)\neq (x_2,y_2)$ debe ser $0<a^2+b^2$ y sólo queda una posibilidad: $$a^2+b^2=p.$$
\end{proof}

Recordamos que como consecuencia del Teorema Fundamental de la Aritmética, todo número natural se puede escribir como producto de un número cuadrado y un entero libre de cuadrados.
En efecto, si $n=\prod_{i=1}^{k}p_i^{e_i}$ es la factorización de $n$ como producto de primos y reordenamos los primos $p_i$ de forma que los los $l$ primeros están elevados a una potencia impar y los demás están elevados a una potencia par, si escribimos $e_i=2f_i+1$ para $1\leq i\leq l$ y $e_i=2f_i$ para $l+1\leq i\leq k$, tendremos que $$n=\prod_{i=1}^{l}p_i^{2f_i+1}\prod_{i=l+1}^{k}p_i^{2f_i}=\prod_{i=1}^{l}p_i\prod_{i=1}^{k}p_i^{2f_i}=n_1n_2^2$$ siendo $n_1:=\prod_{i=1}^{l}p_i$ y $n_2:=\prod_{i=1}^{k}p_i^{f_i}$.

\begin{teo}
	Sea $n$ un número natural, que se escribe como $n=n_1n_2^2$ con $n_1$ libre de cuadrados (lo cual tiene queda justificado por el párrafo anterior). Entonces $n$ es suma de dos cuadrados si y sólo si $n_1$ no tiene ningún factor primo de la forma $p\equiv3\ (mod\ 4)$.
\end{teo}
\begin{proof}[Demostración]
	Supongamos que $n$ es suma de dos cuadrados y probemos que si cierto número primo $p$ divisor de $n$ es de la forma $p\equiv3\ (mod\ 4)$, entonces la máxima potencia de $p$ que divide a $n$ es par, lo cual implica que $p$ no será factor de $n_1$. 
	Supongamos pues que la máxima potencia de $p$ que divide a $n$ es $p^{2e+1}$ y probemos que entonces $p^{2e+2}$ también divide a $n$, lo cual será una contradicción. 
	
	Por inducción sobre $e$:
	
	Si $e=0$, tenemos que $p|n=a^2+b^2$ y entonces $p|a$ y $p|b$, pues si por ejemplo $p$ no dividiese a $b$, $b$ sería coprimo con $p$ y por ser $p$ primo existiría el inverso de $b$ en $\mathbb{Z}_p$ igual a $b^{-1}$. Entonces como $a^2+b^2\equiv0\ (mod\ p)$ esto implica que $$(ab^{-1})^2+1\equiv0\ (mod\ p),$$ es decir,  $$\left( \frac{-1}{p}\right)=1$$ lo cual contradice (\ref{-1 res}).
	
	Supongamos pues que esto es cierto para $e-1$ y probémoslo para $e$. 
	Supongamos que $p^{2e+1}|n$. Entonces $p|n$ y como acabamos de ver esto implica que $p|a$ y $p|b$ y por tanto $$\frac{n}{p^2}=\left(\frac{a}{p}\right)^2+\left(\frac{b}{p}\right)^2$$ pero entonces $p^{2e-1}|\frac{n}{p^2}$ y $\frac{n}{p^2}$ es suma de dos cuadrados, lo cual contradice nuestra hipótesis.
	
	Supongamos ahora que $n$ es tal que $n_1$ no contiene ningún factor primo de la forma $p\equiv3\ (mod\ 4)$ y veamos que $n$ es suma de dos cuadrados. 
	Por (\ref{primo}) cada factor primo de $n_1$ es suma de dos cuadrados, y como vimos al principio de la sección, el producto de dos números que son suma de dos cuadrados es también suma de dos cuadrados. Podemos por tanto asegurar que existen $a,b\in \mathbb{N}$ tales que $a^2+b^2=n_1$ y multiplicando a ambos lados por $n_2^2$ tendremos que $$n=n_1n_2^2=(a^2+b^2)n_2^2= (an_2)^2+(bn_2)^2.$$
\end{proof}
\begin{ej}
	
	Veamos usando GAP cómo determinar si los enteros $123456789$ y $987654321$ son suma de dos cuadrados. 
	
	\begin{lstlisting}[language=GAP]
	
	
	gap> PrintFactorsInt(123456789);
	3^2*3607*3803
	gap> 3607 mod 4;
	3
	
	
	\end{lstlisting}
	
	Vemos que como el entero libre de cuadrados de $123456789$ contiene el primo $3607$ de la forma $p\equiv3\ (mod\ 4)$, no puede expresarse como suma de dos cuadrados.
	
	\begin{lstlisting}[language=GAP]
	
	gap> PrintFactorsInt(987654321);
	3^2*17^2*379721
	gap> 379721 mod 4;
	1
	
	
	\end{lstlisting}
	
	Puesto que el único factor libre de cuadrados de $987654321$ es $379721$ y este es de la forma $p\equiv1\ (mod\ 4)$, se puede expresar como suma de dos cuadrados.
	
\end{ej}

\subsection{Ecuaciones en congruencias de segundo grado}

Dada una ecuación en congruencia lineal del tipo $ax+b\equiv0 \ (mod\ m)$ es muy sencillo determinar cuándo tiene solución (tal y como vimos en el primer tema), sin embargo, la situación se vuelve más compleja cuando estudiamos la ecuación $$ ax^2+bx+c\equiv 0\ (mod\ m).$$

Para determinar si esta ecuación tiene o no solución primero la transformamos en una ecuación más sencilla de la siguiente forma: $$ ax^2+bx+c\equiv 0\ (mod\ m)\Leftrightarrow 4a^2x^2+4abx+4ac\equiv 0\ (mod\ 4am)$$$$\Leftrightarrow (2ax+b)^2 \equiv b^2-4ac \ (mod\ 4am)$$

Por tanto, considerando la ecuación $$y^2\equiv b^2-4ac \ (mod\ 4am)$$ podemos determinar de forma eficiente si tiene o no solución con todos los resultados que hemos tratado en la primera sección, y en caso de tener que cierto $y$ es una solución, bastará con resolver la ecuación lineal $$2ax+b\equiv y\ (mod\ 4am).$$

\subsection{Números Primos}
En esta sección vamos a ver cómo usando la Ley de Reciprocidad Cuadrática podremos demostrar la existencia de infinitos números primos congruentes con $4$ módulo $5$. 

Comenzamos recordando este clásico teorema debido a Euclides.

\begin{teo}[Euclides]
	Existen infinitos números primos.
\end{teo}
\begin{proof}[Demostración]
	Supongamos que sólo hay un conjunto finito de números primos $\{p_1,\ldots,p_n\}$. 
	Sea $p:=p_1p_2\ldots p_n+1$. Tenemos que $p$ no es divisible por $p_i$ para ningún $p_i$ pues si no $p_i|(p- p_1p_2\ldots p_n)=1$. Por tanto $p$ es un entero mayor estricto que $p_i$ para $1\leq i\leq n$ que no es divisible por ningún número distinto de sí mismo y de $1$, lo cual equivale a decir que $p$ es un número primo tal que $$p\notin\{p_1,\ldots,p_n\}$$ y esto supone una contradicción.
\end{proof}

La siguiente proposición generaliza el teorema de Euclides:

\begin{prop}\label{euclides-ext}
	
	Sea $f\in\mathbb{Z}[x]$ un polinomio no constante y $$P_f:=\{p\mid \mbox{ p es primo y } p|f(n) \mbox{ para algún }n\in\mathbb{N} \}.$$ Entonces $P_f$ contiene infinitos enteros.
\end{prop}
\begin{proof}[Demostración]

	Si $f(0)=0$ entonces $f$ no tiene término independiente y para cada primo $p$ se tiene que $p|f(p)$.
	
	Supongamos pues que $f(0)\neq 0$. Si sólo hubieran una cantidad finita pongamos $P_f=\{p_1,\ldots,p_m\}$ tendríamos que, para cualquier entero $k$, $\frac{1}{f(0)}f(kf(0)p_1\ldots p_m)$ es un entero de la forma $n\cdot (p_1\ldots p_m)+1$ para cierto $n\in\mathbb{N}$, lo cual muestra que no es divisible por ningún $p_i$, y si consideramos $k$ lo suficientemente grande tendrá que ser divisible por algún número primo $p_{m+1}$ mayor que $p_i$ para $1\leq i\leq m$, lo cual es una contradicción.
\end{proof}

Como consecuencia de esta proposición, si consideramos un número primo $p$ y el polinomio $f(x)=x^2-p$, tendremos que existen infinitos números primos $q$ tales que $q|n^2-p$ para algún $n\in\mathbb{N}$, es decir, existen infinitos números primos $q$ para los cuales $p$ es un residuo cuadrático.

Vamos ahora con el teorema principal de esta sección: 

\begin{teo} \label{4mod5}
	Existen infinitos números primos de la forma $p\equiv 4\ (mod\ 5)$.
\end{teo}
\begin{proof}[Demostración]
	
	Si definimos $f(x)=5x^2-1\in\mathbb{Z}[x]$ tenemos por la proposición (\ref{euclides-ext}) que existen infinitos números primos que dividan a algún número de la forma $5n^2-1$ con $n\in\mathbb{N}$. Si $p$ es un primo impar tal que $p|5n^2-1$ esto implica que $1^2\equiv 1\equiv 5n^2\ (mod\ p)$ y por tanto que $5n^2$ es un residuo cuadrático módulo $p$, es decir $$\left( \frac{5n^2}{p}\right)=1.$$
	Como obviamente $\left( \frac{n^2}{p}\right)=1$, por el teorema (\ref{prod-residuos}) tenemos que $$\left( \frac{5}{p}\right)=1$$ y como consecuencia de la Ley de Reciprocidad Cuadrática $$\left( \frac{p}{5}\right)=1.$$ Por tanto existe un entero $u\in\mathbb{Z}$ tal que $u^2\equiv  p (mod\ 5)$ y como los números impares cuadrados son o bien congruentes a $1$ o bien a $4$ módulo $5$, $$p\equiv  1\mbox{ o }4\ (mod\ 5)$$ y para terminar la prueba sólo falta comprobar que no puede haber una cantidad finita tal que $p\equiv  4\ (mod\ 5)$. Supongamos que hay una cantidad finita $\{p_1,\ldots,p_m\}$ de primos tales que $p\equiv  4\ (mod\ 5)$ y sea $n:=2p_1\ldots p_m$. Entonces puesto que $5n^2-1$ es impar cualquier divisor primo suyo será congruente con $1$ o $4$ módulo $5$, y como $5n^2-1$ es coprimo con $p_i$ para $1\leq i\leq m$, todos sus divisores primos tienen que ser congruentes con $1$ módulo $5$, pero esto es imposible, porque en ese caso tendría que ser $5n^2-1\equiv  1\ (mod\ 5)$, y se comprueba fácilmente que $5n^2-1\equiv  4\ (mod\ 5)$.
\end{proof} 

Teniendo en cuenta que todos los primos de la forma $p\equiv 4\ (mod\ 5)$ tienen por último dígito $9$, este teorema equivale al enunciado más "visual" que afirma que existen infinitos números primos cuyo último dígito sea $9$.

Antes de terminar la sección, cabría comentar que este es un caso particular del \emph{Teorema de Dirichlet de progresiones aritméticas}, el cual afirma que dados dos enteros positivos y coprimos $a,d\in\mathbb{N}$, existen infinitos números naturales $n\in\mathbb{N}$ para los cuales $a+n\cdot d$ es primo. 

La demostración del teorema de Dirichlet queda fuera de nuestro alcance debido a su gran complejidad, pero muchos casos particulares pueden demostrarse gracias a la Ley de Reciprocidad Cuadrática imitando la demostración del teorema (\ref{4mod5}).

\subsection{Soluciones enteras de ecuaciones elípticas}

Recordamos que una curva elíptica sobre un cuerpo de característica distinta de $2$ y de $3$ es (en su forma simplificada) el conjunto de soluciones de la ecuación $$y^2=x^3+ax+b$$ siendo $a$ y $b$ elementos de dicho cuerpo (i.e. la curva algebraica definida por dicha ecuación).

Una cuestión interesante sobre las curvas elípticas con coeficientes racionales (u otro cuerpo de característica $0$) consiste en determinar si tiene soluciones $(x,y)$ formadas por números enteros y en caso de tenerlas, determinar cuántas pueden haber. 

Sin ahondar demasiado en este asunto, mencionamos que se ha probado que sobre el cuerpo de los números racionales el conjunto de soluciones formadas por números enteros de una curva elíptica dada tiene que ser finito. 
Damos a continuación un ejemplo ilustrativo de cómo puede usarse la Ley de Reciprocidad Cuadrática para determinar si una curva no tiene soluciones enteras. 

\begin{prop}
	La ecuación elíptica $y^2+3=x^3-x$ no tiene soluciones enteras.
\end{prop}
\begin{proof}[Demostración]
	Supongamos que $(x,y)$ es una solución con $x,y\in\mathbb{Z}$. Puesto que $x^3-x$ es siempre par, $y^2+3$ tiene que ser par y por tanto $y$ tiene que ser impar. Supongamos que $y=2k+1$ para cierto $k\in\mathbb{Z}$, entonces $y^2+3=(4k^2+4k+1)+3=4(k^2+k+1)$ y puesto que $k^2+k$ siempre es par, deducimos que $4|y^2+3$ pero $8\not| y^2+3$.
	
	Si $x$ fuera impar, $x^2-1$ sería divisible por $8$ y por tanto $8|x(x^2-1)=y^2+3$, pero acabamos de ver que esto no es posible, luego $x$ tiene que ser par.
	
	Si $x$ es par, tanto $x-1$ como $x+1$ son impares, y puesto que $y^2+3=(x-1)x(x+1)$ deducimos que $4|x$ pero $8\not| x$.
	
	Como $x-1,x,x+1$ son tres enteros consecutivos, uno de ellos tiene que ser congruente con $2$ módulo $3$, y puesto que $(x-1)x(x+1)=4(x-1)\frac{x}{4}(x+1)$ y además por ser $x$ múltiplo de $4$ es $x\equiv \frac{x}{4}\ (mod\ 3)$, tendrá que darse que $(x-1),\frac{x}{4}$ o bien $(x+1)$ será congruente con $2$ módulo $3$. 

	Aquel que sea congruente con $2$ módulo $3$ tendrá que tener como factor algún primo congruente con $2$ módulo $3$ y puesto que  $(x-1),\frac{x}{4}$ y $(x+1)$ son todos impares ese primo $p$ también tendrá que ser impar.
	
	Hemos probado pues que existe un primo impar $p\equiv 2\ (mod\ 3)$ tal que $p|y^2+3$, lo cual equivale a afirmar que $$\left( \frac{-3}{p}\right)=1.$$
	
	Por la Ley de Reciprocidad Cuadrática, tenemos que $$\left( \frac{p}{3}\right)=\left( \frac{-3}{p}\right)=1$$ luego existe un $u\in\mathbb{Z}$ tal que $u^2\equiv p\ (mod\ 3)$ pero como para cualquier entero $u$, $u^2\equiv 0 \mbox{ o }1 \ (mod\ 3)$ esto implicaría que $$p\equiv 0 \mbox{ o }1 \ (mod\ 3),$$ lo cual es una contradicción.
\end{proof}


\begin{thebibliography}{X}
	\bibitem{Dis} C.F.Gauss, "Disquisitiones Arithmeticae"
	\bibitem{Alg} F.Delgado, C.Fuertes y S.Xambó, "Introducción al Álgebra"
	
\end{thebibliography}
\end{document}